\documentclass[11pt,oneside]{amsart}
\usepackage{fullpage}
\usepackage[left=1in, right=1in]{geometry}
\usepackage{amsmath,amsthm,amsfonts,amssymb,amscd}
\usepackage{relsize}

\makeatletter
\let\save@mathaccent\mathaccent
\newcommand*\if@single[3]{%
	\setbox0\hbox{${\mathaccent"0362{#1}}^H$}%
	\setbox2\hbox{${\mathaccent"0362{\kern0pt#1}}^H$}%
	\ifdim\ht0=\ht2 #3\else #2\fi
}
\newcommand*\rel@kern[1]{\kern#1\dimexpr\macc@kerna}
\newcommand*\widebar[1]{\@ifnextchar^{{\wide@bar{#1}{0}}}{\wide@bar{#1}{1}}}
\newcommand*\wide@bar[2]{\if@single{#1}{\wide@bar@{#1}{#2}{1}}{\wide@bar@{#1}{#2}{2}}}
\newcommand*\wide@bar@[3]{%
	\begingroup
	\def\mathaccent##1##2{%
		\let\mathaccent\save@mathaccent
		\if#32 \let\macc@nucleus\first@char \fi
		\setbox\z@\hbox{$\macc@style{\macc@nucleus}_{}$}%
		\setbox\tw@\hbox{$\macc@style{\macc@nucleus}{}_{}$}%
		\dimen@\wd\tw@
		\advance\dimen@-\wd\z@
		\divide\dimen@ 3
		\@tempdima\wd\tw@
		\advance\@tempdima-\scriptspace
		\divide\@tempdima 10
		\advance\dimen@-\@tempdima
		\ifdim\dimen@>\z@ \dimen@0pt\fi
		\rel@kern{0.6}\kern-\dimen@
		\if#31
		\overline{\rel@kern{-0.6}\kern\dimen@\macc@nucleus\rel@kern{0.4}\kern\dimen@}%
		\advance\dimen@0.4\dimexpr\macc@kerna
		\let\final@kern#2%
		\ifdim\dimen@<\z@ \let\final@kern1\fi
		\if\final@kern1 \kern-\dimen@\fi
		\else
		\overline{\rel@kern{-0.6}\kern\dimen@#1}%
		\fi
	}%
	\macc@depth\@ne
	\let\math@bgroup\@empty \let\math@egroup\macc@set@skewchar
	\mathsurround\z@ \frozen@everymath{\mathgroup\macc@group\relax}%
	\macc@set@skewchar\relax
	\let\mathaccentV\macc@nested@a
	\if#31
	\macc@nested@a\relax111{#1}%
	\else
	\def\gobble@till@marker##1\endmarker{}%
	\futurelet\first@char\gobble@till@marker#1\endmarker
	\ifcat\noexpand\first@char A\else
	\def\first@char{}%
	\fi
	\macc@nested@a\relax111{\first@char}%
	\fi
	\endgroup
}
\makeatother
\makeatletter
\renewcommand*\env@matrix[1][*\c@MaxMatrixCols c]{%
	\hskip -\arraycolsep
	\let\@ifnextchar\new@ifnextchar
	\array{#1}}
\def\thm@space@setup{%
	\thm@preskip=\parskip \thm@postskip=0pt
}
\makeatother

\makeatletter
\@namedef{subjclassname@2020}{\textup{2020} Mathematics Subject Classification}
\makeatother

\usepackage{lastpage}
\usepackage{enumerate}
\usepackage{fancyhdr}
\usepackage{mathrsfs}
\usepackage{xcolor}
\usepackage{graphicx}
\usepackage{listings}
\usepackage{hyperref}
\usepackage{euscript}
\usepackage{amsthm}
\usepackage{xifthen}
\usepackage{verbatim}
\usepackage{hyperref}
\usepackage{lipsum}
\usepackage{float}
\usepackage{subfigure}

\usepackage{array}
\newcolumntype{M}[1]{>{\centering\arraybackslash}m{#1}}

\hypersetup{%
	colorlinks=true,
	linkcolor=blue,
	citecolor=blue,
	linkbordercolor={0 0 1}
}

\lstdefinestyle{Python}{
	language        = Python,
	frame           = lines, 
	basicstyle      = \footnotesize,
	keywordstyle    = \color{blue},
	stringstyle     = \color{green},
	commentstyle    = \color{red}\ttfamily
}

\newcommand{\Z}{\mathbb{Z}}

\newcommand{\N}{\mathbb{N}}

\newcommand{\Homology}{\ensuremath{{\sf{H}}}}
\newcommand{\Chain}{\ensuremath{{\sf{C}}}}

\newcommand{\bcap}{\bigcap}
\newcommand{\sseq}{\subseteq}

\theoremstyle{definition}
\newtheorem{theorem}{Theorem}[section]
\theoremstyle{definition}
\newtheorem{lem}[theorem]{Lemma}
\theoremstyle{definition}

\theoremstyle{definition}
\newtheorem{prop}[theorem]{Proposition}
\theoremstyle{definition}
\newtheorem{cor}[theorem]{Corollary}
\theoremstyle{remark}
\newtheorem*{rem}{Remark}
\theoremstyle{definition}
\newtheorem{defin}[theorem]{Definition}
\DeclareMathOperator*{\bigast}{\raisebox{-0.8ex}{\scalebox{2.5}{$\ast$}}}

\newcommand{\s}{\mathfrak{s}}
\renewcommand{\t}{\mathfrak{t}}
\newcommand{\lfrf}[1]{\lfloor #1 \rfloor}
\author[H.~K.~Chong]{Hip Kuen Chong}
\address{Dept. of Math. \& Stats.\\
	McGill Univ. \\
	Montreal, QC, Canada H3A 0B9 }
\email{chonghk1997@gmail.com}
\author[D.~T.~Wise]{Daniel T. Wise}
\email{wise@math.mcgill.ca}
\date{July 27, 2021}
\thanks{Research supported by NSERC}

\title{An Uncountable Family of Finitely Generated Residually Finite Groups}

\begin{document}

\begin{abstract}
	We study a family of finitely generated residually finite groups.  These groups are doubles $F_2*_H F_2$ of a rank-$2$ free group $F_2$ along an infinitely generated subgroup $H$.  Varying $H$ yields uncountably many groups up to isomorphism. 
\end{abstract}

\maketitle

%==========================
%  Introduction
%==========================

\section{Introduction}

A group $G$ is \emph{residually finite} if for each $g\neq 1$, there is a finite quotient $G\to \widebar{G}$ with $\widebar{g}\neq1$.  
Residually finite groups form a privileged and arguably rare class of groups.  It is a routine exercise in small cancellation theory to provide uncountably many isomorphism classes of $2$-generated groups.  For instance, consider the following presentation where $(p_i)$ is a sequence of distinct primes $\geq 7$.  
$$\langle a,b\mid {\left(abab^2\cdots ab^{100i}\right)}^{p_i}\colon i\in \N\rangle$$ 
Since the degrees of torsion element of the group are precisely the elements of $\{p_i\}$, these uncountably many groups are pairwise non-isomorphic.

This paper emerges from our curiosity to produce uncountably many finitely generated residually finite groups.  The family of examples we constructed have a very simple structure: they are all doubles $F*_H F$ of a free group $F$.  Hence, they are particularly familiar examples in combinatorial group theory satisfying many routine properties (e.g.\ torsion-free, cohomological dimension~$2$, fundamental groups of non-positively curved spaces). Though our groups have a very simple structure,
and form a very flexible family, we are not the first to wander in this direction.  

Grigorchuk's family of \emph{Grigorchuk groups} \cite{MR764305} provide uncountably many isomorphism classes of residually finite torsion groups.  This has been revisited in \cite{Benli2015} where uncountably many just-infinite branch pro-$2$ groups are provided.
There are even earlier uncountable classes, again having torsion, which we discuss below.

There is an interesting interplay between residual finiteness and the word problem.  Dyson observed that every finitely presented residually finite group has solvable word problem \cite{DysonNotice}.  On the other hand, Higman showed that every finitely generated recursively presented group embeds as a subgroup of a finitely presented group \cite{MR130286}.  Higman's embedding theorem demonstrates that there are finitely presented groups with unsolvable word problem, as there are finitely generated recursively presented groups having unsolvable word problem.  

It is natural to attempt to relax the finite presentability hypothesis in Dyson's result.  Following Higman's construction, Meskin gave an example of a finitely generated residually finite group with unsolvable word problem \cite{MR335645}.  Around the same time, Dyson constructed a concrete family of examples of finitely generated residually finite groups with unsolvable word problem \cite{MR0360843}.    Rauzy  used Dyson's examples to prove that Higman's embedding theorem cannot be extended to the category of finitely generated residually finite groups with solvable word problem \cite{rauzy2020obstruction}.       
The example we constructed can be recursively presented but still have unsolvable word problem.  We suspect that Dyson's uncountable family actually provides uncountably many isomorphism classes.  
Dyson's examples are also doubles $L*_H L$, but of a lamplighter group instead of a free group.  Hence, they are rich in torsion like Grigorchuk's examples.  

By controlling subgroup growth, Pyber produced uncountably many pairwise non-isomorphic 4-generator residually finite groups with isomorphic profinite completions \cite{MR2031168}.  By counting irreducible representations, Kassabov{\textendash}Nikolov also produced an uncountable family of pairwise non-isomorphic residually finite groups \cite{MR2264720}.  More recently, 
Segal provides a center-by-metabelian family, and Nikolov{\textendash}Segal provide uncountably many non-isomorphic length-$4$ soluble groups that have the same profinite completion \cite{SegalCentralMetababelian,NikolovSegalSoluble}.  

%Dyson predicted that her construction of groups that might lead to future applications, yet it took over $40$ years until these examples were utilized recently by Rauzy.   We can hope our flexible family of examples might also serve some future applications.

Section~\ref{sec:Gs} describes the groups $G_\s$. 
Section~\ref{sec:rf} proves they are residually finite.
Section~\ref{sec:homology} computes  homologies of certain quotients of $G_\s$.
Section~\ref{sec:nonisomorphic} uses this to show that these groups are pairwise non-isomorphic.
Section~\ref{sec:wordproblems} notes the simple relationship to unsolvable word and membership problems.

%==========================
%  Main Result
%==========================

\section{The groups $G_\s$}\label{sec:Gs}

Here we describe our family of finitely generated groups $G_{\s}$.

\begin{defin}
	Let $F$ be a group with an isomorphic copy $\widebar{F}$.  Let $H\leq F$ and $\widebar{H}\leq \widebar{F}$ be corresponding subgroups.  The \emph{double of $F$ along $H$} is the amalgamated product $G \ = \ F*_{H= \widebar{H}}\widebar{F}$.
\end{defin}

%Our groups $G_{(s)}$ are doubles of a free group along a subgroup $H_{(s)}$  depending on a sequence $(s_n)_{n=0}^\infty$.  

\begin{defin}
	A sequence $\s=(s_n)_{n=0}^\infty$ is \emph{multiplying} if $s_n$ and $s_{n+1} / s_n$ are positive integers for all $n$, and $\lim s_n = \infty$.
	%Note that multiplying sequence  are non-decreasing.
\end{defin}

Let $\s$ be multiplying.
Let $F=\langle a,b\rangle$ be a rank $2$ free group.  Let $H_{\s}\leq F$ be the subgroup $H_{\s}=\langle b^na^{s_n}b^{-n}:n\geq 0\rangle$.  We define $G_{\s}$ as the double of $F$ along $H_{\s}$: 
$$G_{\s} \ = \ F *_{H_{\s} = \widebar{H}_{\s}} \widebar{F}.$$

\section{Residual Finiteness of $G_\s$}\label{sec:rf}
%==========================
%  Residually Finite
%==========================
In this section we prove residual finiteness of $G_\s$ by generalizing \cite[Lem.2]{lemmaproof}.

A subgroup $H \leq G$ is \textit{closed} (in the profinite topology) if $H$ is the intersection of finite index subgroups of $G$.  For instance, $\{1\}\sseq G$ is closed iff $G$ is residually finite.  We use below that every finitely generated subgroup of a free group $F$ is closed \cite[Thm.5.1]{Marshallsubgroup}.  In particular, $F$ is residually finite. 
Lemma~\ref{doublerf} is used to prove that $G_{\s}$ is residually finite. The proof of Lemma~\ref{doublerf} is a special case of a result of G. Baumslag, see for instance, \cite{lemmaproof}.  

\begin{lem}\label{doublerf}
	Let $H$ be a closed subgroup of a residually finite group $F$.  Then the double $F*_{H= \widebar{H}}\widebar{F}$ is also residually finite.
\end{lem}

If $H_{\s}$ is a closed subgroup of a free group $F$, then since $F$ is residually finite, $G_{\s}$ is also residually finite 
by Lemma~\ref{doublerf}, since $G_{\s}$ is the double of $F$ along $H_{\s}$.  Hence, 
it suffices to prove Lemma~\ref{lem:Closed} to show that $G_{\s}$ is residually finite.
The proof  uses the graphical viewpoint on subgroups of free groups
as popularized in combinatorial group theory by
Stallings  \cite{MR695906} and \cite[Rem.3.6]{danigroup}.  

\begin{defin}
	The \emph{based core at $x$} of a connected graph $B$  is the smallest connected subgraph containing $x$
	and all closed cycles of $B$.
\end{defin}

\begin{lem}\label{lem:Closed}
	If $\s$ is multiplying  then $H_{\s}$ is a closed subgroup of $F$.
\end{lem}
\begin{proof}
	As  above, every finitely generated subgroup of a free group is closed.  Hence, it suffices to show that $H_{\s}$ is the intersection of a sequence of finitely generated subgroups of $F$.  
	
	For each $k\geq 0$, define $S_k\leq F$ as follows:
	$$S_k\ =\ \langle b^k, b^ja^{s_j}b^{-j}:0\leq j<k\rangle.$$
	
To show $H_{\s} \ \sseq \ \bcap S_k$, it suffices to show that $b^ma^{s_m}b^{-m}\in S_k$ for  $m,k\geq 0$.
Fix $m, k\geq 0$.  If $m<k$, then $b^ma^{s_m}b^{-m}\in S_k$ by definition.  If $m\geq k$, let $m=qk+r$ for integers $q,r\geq 0$ with $0\leq r<k$.  Let $f=s_m / s_r$.  Since $\s$ is multiplying, $f$ is a positive integer.  Thus: 
%$H_{\s}\sseq \bcap S_k$ since
	\begin{align*}
		b^ma^{s_m}b^{-m}\ &=\ b^{qk+r}a^{s_rf}b^{-qk-r}
		=\ (b^k)^q (b^ra^{s_r}b^{-r})^{f} (b^{-k})^q\ \in\ S_k.
	\end{align*}

	It remains to show that $\bcap S_k \sseq H_{\s}$.  
	Consider $w\in \bcap S_k$.  Since $\lim s_n= \infty$, we may choose $k$ with $k>2|w|$ and $s_{k-|w|}>|w|$.  The based core $C$ of the covering space associated with $S_k$ is shown in Figure~\ref{Hk_basedcore}.  The graph $C$ is the union of a based length~$k$ cycle of $b$-edges, and a length~$s_i$ cycle of $a$-edges attached at its $i$-th vertex for $0\leq i<k$.
	
	\begin{figure}[h]
		\centering
		\includegraphics[width=0.5\textwidth]{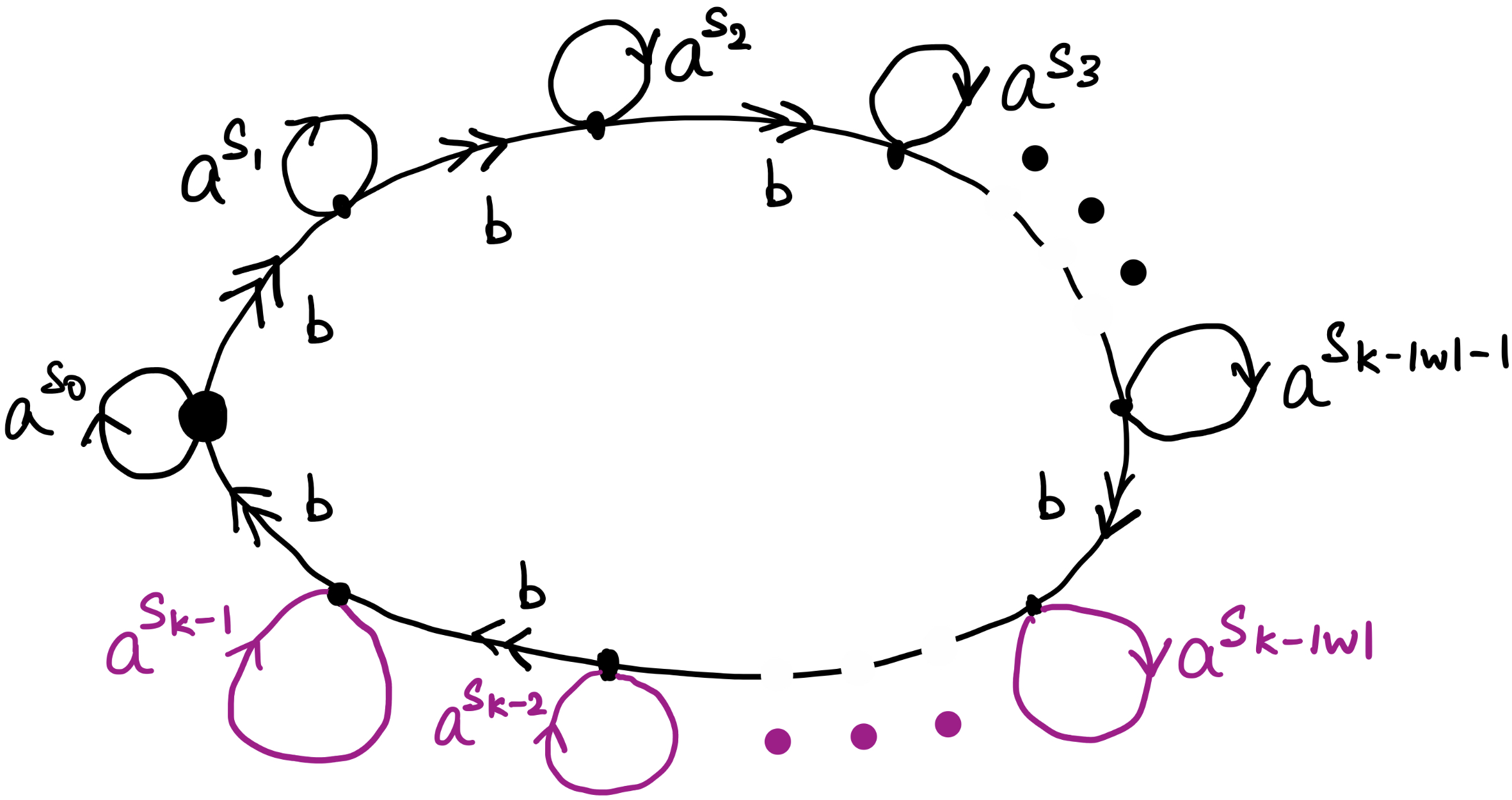}
		\caption[The based core $C$ of the covering space associated with the subgroup $S_k$]{The core $C$ of the based covering space associated with $S_k$. }
		\label{Hk_basedcore}
	\end{figure}
	
	Let $\omega$ be the immersed loop in $C$ labelled by $w$.  Then, $\omega$ is shorter than the $a^{s_{k-m}}$ (purple) loops for $1\leq m\leq |w|$, because $s_{k-|w|}>|w|$.  Thus,  $\omega$ cannot traverse an edge of those (purple) loops.  Therefore, $\omega$ is  loop immersed in the  subgraph of $C$ obtained by removing those (purple) loops.  
	
	Moreover, $\omega$ cannot traverse an edge of $b^{|w|}$ at the bottom of the graph since $\omega$ cannot backtrack.  
	
	Therefore, $\omega$ is immersed in the subgraph $C'$ obtained by discarding edges which $\omega$ cannot traverse as explained above. See Figure~\ref{Hk_basedcore_refined}.  Since $C'$ is a based subgraph of the based core of the covering space associated with $H_{\s}$, we see that $w\in H_{\s}$.  

%		Therefore, $H_{\s}=\bcap S_k$ and thus $H_{\s}$ is a closed subgroup of $F$.
\end{proof}
	
\begin{figure}[h]
	\centering
	\includegraphics[width=0.6\textwidth]{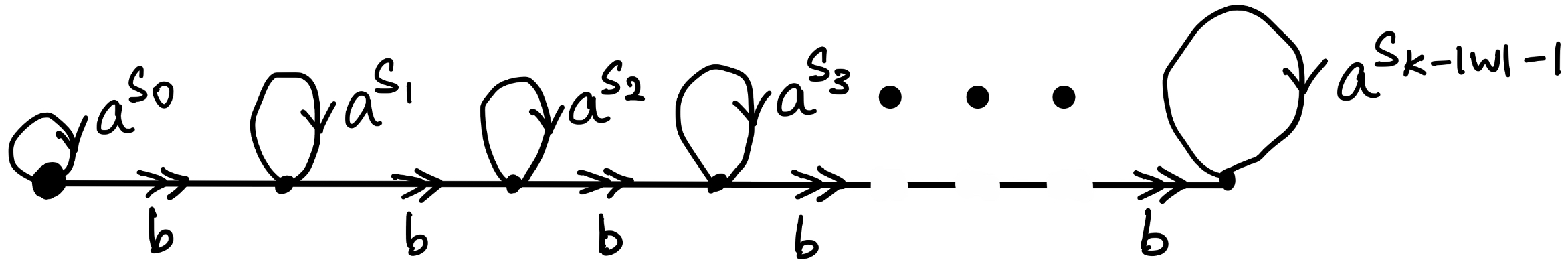}
	\caption[Graph $C'$ where the loop $\omega$ could possibly immerse]{The based graph $C'$ after removing the (purple) loops $a^{s_{k-m}}$ for $1\leq m\leq |w|$ and the bottom $b^{|w|}$ path in Figure~\ref{Hk_basedcore}. The base point is bold.}
	\label{Hk_basedcore_refined}
\end{figure}

\begin{rem}
It is not difficult to show that $G_\s$ is residually a finite $p$-group if and only if all elements of $\s$ are powers of $p$.

Furthermore, $G_\s$ fails to be residually finite without the requirement that $\s$ is nonbounded.
\end{rem}

\section{The homologies of quotients of $G_\s$}
\label{sec:homology}

We first prove a simple homological lemma.  Let $\Z_n = \Z/n\Z$.

\begin{lem}\label{Homology}
	Let $\s$ be multiplying.  Let $m$ be a positive integer. 
Let $\widehat G_{\s}=G_{\s}/\langle\!\langle a^{m},\widebar{a}^{m}\rangle\!\rangle$.
Then $\Homology_1(\widehat G_{\s})\cong \Z^2\times \Z_{m}\times \Z_{\lfrf{m,s_0}}$ and $\Homology_2(\widehat G_{\s})\cong \bigoplus_{j=1}^\infty \Z_{m / \lfrf{m,s_j}}$, where $\lfrf{a,b} = \gcd(a,b)$.
\end{lem}

\newcommand{\ab}{\textup{\textsf{ab}}}

\begin{proof}
	Recall that $G_{\s}=F*_{H_{\s}= \widebar{H}_{\s}}\widebar{F}$.  Images of subgroups in $G_{\s}\to\widehat G_{\s}$ are hatted. 
	\begin{align}
	\Homology_1(\widehat{G}_{\s})\ \cong \ (\widehat{G}_{\s})_\ab \ &=\ \langle a,b,\widebar{a},\widebar{b}\mid H_{\s}=\widebar{H}_{\s}, \ a^{m},\widebar{a}^{m}\rangle_\ab\nonumber\\ 
	\ &=\ \langle a,b,\widebar{a},\widebar{b}\mid a^{s_0}=\widebar{a}^{s_0},a^{m},\widebar{a}^{m}\rangle_\ab\nonumber\\
	\ &=\ \langle a,b,\widebar{a},\widebar{b}\mid a^{\lfrf{m,s_0}}=\widebar{a}^{\lfrf{m,s_0}},a^{m},\widebar{a}^{m}\rangle_{\ab}\nonumber\\
%	\ &=\ \langle b,\widebar{b}\rangle_\ab\times \langle a\rangle / \langle a^{m}\rangle \times \langle a\widebar{a}^{-1}\rangle /\langle (a\widebar{a}^{-1})^{\lfrf{m,s_0}}\rangle\nonumber\\
	\ &\cong\  \Z^2\times \Z_{m}\times \Z_{\lfrf{m,s_0}}\nonumber\\
	\Homology_1(\widehat{F})\ =\ (\widehat{F})_\ab \ &=\  \langle a,b\mid a^{m}\rangle_\ab\ \cong\  \Z \times \Z_{m} \label{eq:H1F}
	\end{align}
	
We shall compute $\Homology_2(\widehat G_{\s})$  via the Mayer{\textendash}Vietoris sequence for amalgamated products. 
Firstly, we claim $\widehat G_\s$ decomposes as a double as follows:
	\begin{itemize}
	\item[$-$] $\widehat G_{\s}  \ \cong  \ Q *_{K=\widebar{K}}\widebar{Q}$, where $Q=F/\langle\!\langle a^{m}\rangle\!\rangle = \widehat F$ and $\widebar{Q}=\widebar{F}/\langle\!\langle \widebar{a}^{m}\rangle\!\rangle = \widehat{\widebar{F}}$
	\item[$-$] and where  $K = \operatorname{im}(H_\s)$ under $F\to Q$,
	and $\widebar{K} = \operatorname{im}(\widebar{H}_\s)$ under $\widebar{F}\to \widebar{Q}$. 
	\end{itemize} 
	And hence  $K\to \widehat{H}_\s$ and $\widebar{K}\to \widehat{\widebar{H}}_\s$ are isomorphisms.
	
	The map $\widehat G_{\s}\to Q *_{K=\widebar{K}}\widebar{Q}$ is induced by $F\to Q$ and $\widebar{F}\to \widebar{Q}$.  The retraction maps $G_{\s}\to F$ and $G_{\s}\to \widebar{F}$ project to retractions $\widehat{G}_{\s}\to Q$ and $\widehat{G}_{\s}\to \widebar{Q}$.  Moreover, the retractions $G_{\s}\to F$ and $G_{\s}\to \widebar{F}$ projects to $\widehat{G}_{\s}\to \widehat{F}$ and $\widehat{G}_{\s}\to \widehat{\widebar{F}}$, respectively.  The inclusions $Q\hookrightarrow \widehat{G}_{\s}$ and $Q\hookrightarrow \widehat{G}_{\s}$ induce $Q *_{K=\widebar{K}}\widebar{Q}\to \widehat G_{\s}$, which provides the inverse. Thus $\widehat G_{\s}\cong Q *_{K=\widebar{K}}\widebar{Q}$. 
	%(This is a special case of a general statement of quotient of amalgamated products.)

	To calculate $\Homology_1(\widehat H_\s)$, we first determine $\widehat H_\s$
	through the following sequence of isomorphisms. 
The first holds by declaring $\hat a,\hat b$ to be the images of $a,b$. 
The second holds by the Normal Form Theorem for free products \cite[Thm.IV.1.2]{LS77}.
The third holds for each factor as $s_j$ has order  $m / \lfrf{s_j,m}$ in $\Z/m\Z$.
\begin{equation}\label{eq:LabelOfFreeProduct}
	\widehat{H}_{\s}
	\ \ = \ \
	\langle \hat b^{j} \hat  a^{s_j} \hat  b^{-j}\colon j\geq 0\rangle  
	\ \ \cong \ \
	\bigast_{j=0}^\infty \langle \hat b^{j} \hat  a^{s_j}  \hat b^{-j} \rangle
	\ \ \cong \ \
	\bigast_{j=0}^\infty \Z_{m / \lfrf{s_j,m}}.
\end{equation}
%
% Any  nontrivial $g\in \bigast_{j=0}^\infty \langle b^{j}a^{s_j}b^{-j}\rangle/\langle\!\langle a^{m}\rangle\!\rangle$ can be written as	$$g\ =\ (b^{j_1}a^{k_1s_{j_1}}b^{-j_1})(b^{j_2}a^{k_2s_{j_2}}b^{-j_2})\cdots (b^{j_\ell}a^{k_\ell s_{j_\ell}}b^{-j_\ell})\ =\ b^{j_1}a^{k_1s_{j_1}}b^{j_2-j_1}a^{k_2s_{j_2}}\cdots a^{k_\ell s_{j_\ell}}b^{-j_\ell}$$ 	where $m\nmid k_is_{j_i}$ for all $i$ and $j_p\neq j_q$ for $|p-q|=1$.  Hence, the above expression is reduced in $\widehat{H}_{\s}$, so	$\Phi(g)\neq 1_{\widehat H_{\s}}$ by the Normal Form Theorem of Free Products \cite[Thm.IV.1.2]{LS77}.
%	
Therefore, as homology is summable over free products we have:
	\begin{equation}\label{eq:H1Hs}
\Homology_1(\widehat{H}_{\s})
\ \cong \
		% \bigoplus_{j=0}^\infty \Homology_1(\langle b^ja^{s_j}b^{-j}\rangle / \langle\!\langle a^{m}\rangle\!\rangle)\ \cong\
 \bigoplus_{j=0}^\infty \Z_{m / \lfrf{m,s_j}}.
	\end{equation}
	
	To compute  $\Homology_2(\widehat F)$ we  construct an aspherical complex $X$ with $\pi_1 X \cong \widehat F$, so $\Homology_2(\widehat F)=\Homology_2(X)$ \cite[Prop.2.4.1]{aspherical}.  Let $X^2$ be the standard 2-complex
	of $\langle a,b \mid b^m \rangle$.
%Figure~\ref{G1_complex} for $X^2$.
The aspherical complex $X$ is formed by adding appropriate $n$-cells to $X^2$ for $n\geq 3$.  
%
%maybe not used:
%Specifically, $X$ equals $A \vee \langle b\rangle$ where $A$ is a $K(\pi,1)$ for $\Z/m\Z$ obtained by adding higher cells to the complex for $\langle a \mid a^{m} \rangle$.
%
%	\begin{figure}[h]
%		\centering
%		\includegraphics[width=0.3\textwidth]{G1_complex}
%		\caption[Complex $X$ with $\pi_1 X = \widehat F$]{Complex $X^2$ with $\pi_1 X= \widehat F$.  It is the $2$-skeleton of $X$.  There is a $2$-cell attached to the $1$-cell labelled $a$ with boundary $a^{m}$.}
%		\label{G1_complex}
%	\end{figure}
	
$\partial_2: \Chain_2(X)\to \Chain_1(X)$ is injective since the attaching map of the 2-cell has degree~$m$.
Thus
	\begin{equation}\label{eq:H2F}
		\Homology_2(\widehat F) \ =\ \Homology_2(X)\ =\ \frac{\ker \partial_2}{\operatorname{im} \partial_3}\ =\ 0.
	\end{equation}
	
	The following exact sequence is deduced from the Mayer{\textendash}Vietoris sequence for amalgamated products \cite[Prop.2.7.7]{aspherical}. 
	$$\Homology_2(\widehat F)\oplus \Homology_2(\widehat{\widebar{F}})\ \longrightarrow\ 
	\Homology_2(\widehat G_{\s})\ \longrightarrow \ 
	\Homology_1(\widehat H_{\s})\ \longrightarrow \ 
	\Homology_1(\widehat F)\oplus \Homology_1(\widehat{\widebar{F}})$$

	Applying Equations~\eqref{eq:H1F}, \eqref{eq:H1Hs} and \eqref{eq:H2F}, we  obtain the following exact sequence.
	$$0\ \longrightarrow \ 
	\Homology_2(\widehat G_{\s})\  \longrightarrow \ 
	\bigoplus_{j=0}^\infty \Z_{m / \lfrf{m,s_j}}
 	\ \overset{\phi}{\longrightarrow} \ 
 	(\Z \times \Z_{m})^2$$
	
	The map $\phi\colon \Homology_1(\widehat H_{\s})\to \Homology_1(\widehat F)\oplus \Homology_1(\widehat{\widebar{F}})$ is induced by $\widehat H_{\s}\hookrightarrow \widehat F$ and $ \widehat{\widebar{H}}_{\s}\hookrightarrow \widehat{\widebar{F}}$.  
	By Equation~\eqref{eq:LabelOfFreeProduct}, $\phi([\hat b^{j} \hat  a^{s_j}  \hat b^{-j}]) = ([\hat a^{s_j}],[\hat{\bar{a}}^{s_j}])$.  Since $\{[\hat b^{j} \hat  a^{s_j}  \hat b^{-j}]\colon j\geq 0\}$ generates $\Homology_1(\widehat H_{\s})$, $\text{im}(\phi)$ equals $\lfrf{m,s_0}\Z / m\Z\cong \Z_{m / \lfrf{m,s_0}}$ as $s_0\mid s_j$ for $j\geq 0$.
%	
%Let $\Z_{\frac{m}{\lfrf{m,s_j}}}=\langle g_j\rangle$.
%Let the two $\Z_{m}$ components corresponding to $a$ and $\widebar{a}$ have generators $u_1$ and $u_2$ respectively.  Then $\phi(g_j) = (0,s_ju_1,0,s_ju_2)\in (\Z\times \Z_{m})^2$.  
%	
	Therefore, by the first isomorphism theorem, we reach the conclusion:
	\begin{equation*}
		\Homology_2(\widehat G_{\s})
		\ =\ \ker \phi 
		\ \cong\ \frac{\bigoplus_{j=0}^\infty \Z_{m / \lfrf{m,s_j}}}{\Z_{m / \lfrf{m,s_0}}} 
		\ =\ \bigoplus_{j=1}^\infty \Z_{m / \lfrf{m,s_j}}.\qedhere
	\end{equation*}
\end{proof}

\section{Canonical Subgroups and pairwise non-isomorphism}\label{sec:nonisomorphic}
We shall first identify certain canonical subgroups of $G_{\s}$.
We then use this to show that $G_{\s} \not\cong G_{\t}$  for multiplying sequences $\s\neq \t$. 

\begin{defin}
	Element $g\in G$ is \emph{loose} if $C=\langle g\rangle$ is a maximal cyclic subgroup and there exists a maximal cyclic subgroup $C'$ such that $C \cap C'$ is a proper finite index subgroup of $C$.
\end{defin}

%\begin{rem}
%	If $g\in G$ is loose, then $hgh^{-1}$ is loose for any $g\in G$.  
%\end{rem}

\begin{prop}\label{prop:main}
%Let $\s$ be a multiplying sequence.  
If  $\s \neq (1)$,
then $g\in G_{\s}$ is loose if and only if $g$ is conjugate to $a^{\pm1}$ or $\widebar{a}^{\pm1}$.
\end{prop}
\begin{proof}
Choose $p$ with $s_p>1$.  Let $C = \langle b^pab^{-p}\rangle$ and $C' = \langle \widebar{b}^p\widebar{a}\widebar{b}^{-p}\rangle$.
	Then $C\cap C' = \langle b^pa^{s_p}b^{-p}\rangle = \langle \widebar{b}^p\widebar{a}^{s_p}\widebar{b}^{-p}\rangle$ is a proper finite index subgroup of $C$ and $C'$.  Hence, $b^pab^{-p}$ and $\widebar{b}^p\widebar{a}\widebar{b}^{-p}$ are loose and thus $a$ and $\widebar{a}$ are loose, since looseness is preserved by conjugation.
	
	For the other direction, consider the action of $G_{\s}$ on the Bass{\textendash}Serre tree.  Suppose $\langle g\rangle \neq \langle k\rangle$ are maximal cyclic subgroups with $\langle g\rangle\cap \langle k\rangle$ infinite.  Note that either both are hyperbolic or both are elliptic.  Indeed, if $g$ is hyperbolic, then $k$ is hyperbolic since $k^n\in \langle g\rangle$ for some $n$.  
	
	Suppose $g$ and $k$ are hyperbolic, then since $\langle g\rangle$ and $\langle k\rangle$ are commensurable, they stabilize the same axis in the tree and hence act on the axis in the same way.  By maximality of $\langle g\rangle$ and $\langle k\rangle$, $\langle g\rangle = \langle k\rangle$ which contradicts our assumption.
%	
%Hence, $\langle g\rangle \cap \langle k\rangle = \langle g\rangle = \langle k\rangle$ is not a proper subgroup of $%\langle g\rangle$ or $\langle k\rangle$.  
	
Suppose $g$ and $k$ are elliptic.		
	Let $u$ and $v$ be the vertices stabilized by $g$ and $k$, and let $A$ be the geodesic from $u$ to $v$.
	Note that $u\neq v$ since free groups have no loose subgroups.
	%Indeed, distinct maximal cyclic subgroups of a free group have trivial intersection.
	Note that the pointwise-stabilizer of  $A$ is $\langle g\rangle \cap \langle k\rangle$. 
	
	Let $A =e_1e_2\cdots e_m$, with vertices $u=u_0,u_1,\ldots,u_m=v$.  Since $g$ stabilizes $u_0$, there exists $\ell$ such that $g$ stabilizes $u_\ell$ but not  $e_{\ell+1}$.  Thus $g$ is in the vertex group associated to $u_\ell$ but not in the edge group associated to $e_\ell$. However $g^n$ is in the edge group associated to $e_\ell$ for some $n>0$.
	
	It thus suffices to show that $H < F$ has the property that if $g^r\in H$ for some $r$ but $g\not\in H$, then $g$ is conjugate into $\langle a\rangle$.  Equivalently, we show that if a path $S$ in the graph associated to $H$ has a power $S^m$ which is closed, then either $S$ is closed or $S$ is a conjugate of a power of $a$.
	
	Observe that the $b$-exponent sum of $S$ equals $0$, for otherwise, no $S^r$ is closed.  Assume without loss of generality that $S$ is cyclically reduced, and moreover, if $S$ is not a power of $a$, then $S$ starts with $b$.  We then deduce that $S$  starts on the $b$ ray, for otherwise $S^r$ embeds in a tree attached to the interior of an $a$-loop.  If $S$ ends on the $b$ ray, then $S$ is closed since the $b$-exponent sum of $S$ is $0$.  Finally, if $S$ ends in the interior of an $a$-loop, then as above, $S^{r-1}$ embeds in a tree attached to the interior of an $a$-loop, so $S^r$ is not closed.
	%Observe that if $S$ is not closed but starts with $b$, then (need clarification) $S^2$ leaves the graph $H$ in to a tree in the associated covering space.  Hence no $S^m$ can be closed.
\end{proof}

\begin{cor}\label{cor:main}
	Each  $G_{\s}$ is finitely generated and residually finite.  Moreover,
	$G_{\s}\not\cong G_{\t}$	when $\s\neq \t$.
\end{cor}
\begin{proof}
Let $\s$ and $\t$ be multiplying sequences.  Suppose $\phi\colon G_{\s}\to G_{\t}$ is an isomorphism.
 Denote generators for $G_{\s}$ as $a,b,\widebar{a},\widebar{b}$ and that for $G_{\t}$ as $a',b',\widebar{a}',\widebar{b}'$.  Denote $A_m= \langle\!\langle a^{m},\widebar{a}^{m}\rangle\!\rangle\sseq G_{\s}$ and $A_m' = \langle\!\langle a^{m},\widebar{a}^{m}\rangle\!\rangle\sseq G_{\t}$.  
	
	By Proposition~\ref{prop:main}, loose elements $a^{\pm1},\widebar{a}^{\pm1}$ can be mapped to conjugates of $(a')^{\pm1}$ or $(\widebar{a}')^{\pm1}$.  For $s_0=1$, since $a=\widebar{a}$, so $\phi(A_m)= A_m'$ for all $m$.  For $s_0>1$, by abelianization of $G_{\s}$ and $G_{\t}$, $\phi(a)$ is conjugate to $(a')^{\pm1}$ and $\phi(\widebar{a})$ is conjugate to $(\widebar{a}')^{\pm1}$; or $\phi(a)$ is conjugate to  $(\widebar{a}')^{\pm1}$ and $\phi(\widebar{a})$ is conjugate to $(a')^{\pm1}$.  Therefore, $\phi(A_m)= A_m'$ for all $m$.
	
	Suppose sequences $\s\neq \t$ first differ at index $k$, i.e.\ $s_j = t_j$ for $j<k$ and $s_k\neq t_k$.  Without loss of generality let $s_k>t_k$.  Using $s_j\mid s_k$ for $j\leq k$, by Lemma~\ref{Homology}, 
	\begin{align*}
		\Homology_1(G_{\s} / A_{s_k})\ &\cong\ \Z^2\times \Z_{s_k} \times \Z_{s_0}&\Homology_2(G_{\s} / A_{s_k})\ &\cong\ \bigoplus_{j=1}^{k-1} \Z_{s_k / s_j}\\
		\Homology_1(G_{\t} / A_{s_k}')\ &\cong\ \Z^2\times \Z_{s_k} \times \Z_{\lfrf{s_0,t_0}}&\Homology_2(G_{\t} / A_{s_k}')\ &\cong\ \left[\bigoplus_{j=1}^{k-1} \Z_{s_k / s_j}\right] \oplus \left[\bigoplus_{j=k}^\infty \Z_{s_k / \lfrf{s_k,t_j}}\right].
	\end{align*}

	If $k=0$, then $\Z_{\lfrf{s_0,t_0}}=\Z_{t_0}$.  Thus $\Homology_1(G_{\s} / A_{s_k})\not\cong \Homology_1(G_{\t} / A_{s_k}')$.  This is impossible.
	
	If $k\geq 1$, then $\lfrf{s_k,t_k}<s_k$.  Thus $\bigoplus_{j=k}^\infty \Z_{s_k / \lfrf{s_k,t_j}}\not\cong 1$, so $\Homology_2(G_{\s} / A_{s_k})\not\cong \Homology_2(G_{\t} / A_{s_k}')$. This is impossible.
\end{proof}

\section{Most $G_\s$ have unsolvable word problem}\label{sec:wordproblems}

\begin{defin}
	A group $G$ generated by a finite set $S$ has \emph{solvable word problem} if there is a computer program that can determine whether a word in $S$ represents $1_G$.  
	A subgroup $H$ of $G$ has \emph{solvable membership problem} if there is a computer program that determines  whether or not a word in $S$ represents an element in $H$.  
	A sequence $\s$ is \emph{computable} if there is a computer program that outputs $s_n$ for input $n$.  
	\emph{Unsolvable} means not solvable and \emph{uncomputable} means not computable.
\end{defin}

\begin{cor}
	There are uncountably many finitely generated residually finite groups with unsolvable word problem.
\end{cor}
\begin{proof}
	Consider groups $G_{\s}$.  We claim that if $\s$ is an uncomputable sequence, then $H_{\s}$ has unsolvable membership problem.  Assume there is a program that determines whether  $(b^{n}a^pb^{-n}) \in H_\s$.
 Then enumerating $p$
 %determining whether $(b^{n}a^{s}b^{-n})(\widebar{b}^{n}\widebar{a}^{s}\widebar{b}^{-n})^{-1}$ is trivial
	 gives a  program that outputs $s_n$ on input $n$.
	 
	We now show $G_{\s}$ has unsolvable word problem if $H_{\s}$ has unsolvable membership problem.
	   Assume there is a program that determines whether a word is trivial in $G_{\s}$. Consider the involution $\phi\colon G_{\s}\to G_{\s}$ with $\phi(a)=\widebar{a}$ and $\phi(b) = \widebar{b}$.  Then determining whether the word $w\phi(w)^{-1}$ represents $1_{G}$ for any word $w$ in $\{a,b\}$ solves the membership problem for $H$.
	
It thus suffices to prove  there are uncountably many uncomputable multiplying sequences $\s$.
There are  countably many computable sequences since there are countably many (finite) computer programs.
However, there are uncountably many multiplying sequences, e.g.\ let $s_{n+1} / s_n\in \{2,4\}$.
\end{proof}

%\begin{comment}
%	Probably get rid of the following in published version.
%\end{comment}
%
%Given the structure of the group $G_{\s}$ above, we sidestepped the problem of determining $Out(G_{\s})$ (which appears to be isomorphic to $(\Z_2)^3$).  %Probably put it in earlier paragraphs to say "we could do it but we are going to mention it"

\textbf{Acknowledgement:}  We are grateful to Piotr Przytycki, Dan Segal and the referee for helpful comments.  

%\textbf{Question}: Is the outer automorphism group $\operatorname{Out}(G_{\s})=(\Z_2)^3$, generated by (i) $a\mapsto a^{-1}$, (ii) $a\mapsto a'$ and (iii) $b\mapsto b'$?  In order for $\operatorname{Out}(G_{\s})=(\Z_2)^3$, is it necessary to impose stronger assumption to the sequences $(s_n)_{n=0}^\infty$?

%======== Bibliography ========
\bibliographystyle{alpha}
\bibliography{Citation}
\end{document}